\documentclass{amsart}
        \usepackage{latexsym}
        \usepackage{amssymb}
        \usepackage{amsmath}
        \usepackage{amsfonts}
        \usepackage{amsthm}
        \usepackage[hypertexnames=false]{hyperref}
        \ifpdf
          \usepackage[all,knot,poly,2cell]{xy}
        \else
          \usepackage[all,knot,poly,2cell,dvips]{xy}
        \fi
             \UseAllTwocells
        \usepackage{xspace}

        \usepackage{eucal}

        \theoremstyle{plain}
        \newtheorem{theorem}{Theorem}[section]
        \newtheorem{corollary}[theorem]{Corollary}
        \newtheorem{lemma}[theorem]{Lemma}
        
        \newtheorem{proposition}[theorem]{Proposition}
        \newtheorem{maintheorem}{Theorem}

        \theoremstyle{definition}
        \newtheorem{definition}[theorem]{Definition}
        \newtheorem{example}[theorem]{Example}
        
        \newtheorem*{example*}{Example}

        \theoremstyle{remark}
        \newtheorem{remark}[theorem]{Remark}
        \newtheorem*{remark*}{Remark}  
        \newcommand{\suchthat}{\,:\,}
        \newcommand{\itemref}[1]{\eqref{#1}}
        
        \newcommand{\loccit}{[\emph{loc.\ cit.}]\xspace}

        \newcommand{\mathscript}{\mathcal}

        \newcommand{\Z}{\mathbb{Z}}

        \newcommand{\Q}{\mathbb{Q}}

        \newcommand{\Orb}{\mathcal{O}}   
        \newcommand{\red}[1]{{#1}_{\mathrm{red}}} 

        \DeclareMathOperator{\im}{im} \DeclareMathOperator{\Hom}{Hom}

        \DeclareMathOperator{\supp}{supp}
        \newcommand{\COHO}[1]{\mathcal{H}^{{#1}}}
        
        \newcommand{\RDERF}{\mathsf{R}}
        \newcommand{\LDERF}{\mathsf{L}}
        
        \newcommand{\DCAT}{\mathsf{D}}
        
        \newcommand{\SHom}{\mathcal{H}om}
        
        \DeclareMathOperator{\supph}{supph} 
        \newcommand{\QCOH}{\mathsf{QCoh}}
        
        \newcommand{\PERF}{\mathsf{Perf}}

        \newcommand{\ID}[1]{\mathrm{Id}_{#1}}

        \newcommand{\tensor}{\otimes}
        
        \newcommand{\homotopic}{\simeq}

        \newcommand{\opp}{\circ}
        
\renewcommand{\subset}{\subseteq}
\numberwithin{equation}{section}

\newcommand{\qcsubscript}{\mathrm{qc}} 

\newcommand{\DQCOH}[1][]{\DCAT_{\qcsubscript{#1}}} 

\newcommand{\QCPSH}[1]{(#1_{\qcsubscript})_*} 
\newcommand{\QCPBK}[1]{#1^*_{\qcsubscript}}

\newcommand{\shv}[1]{\mathcal{#1}} 
\newcommand{\cplx}[1]{\shv{#1}} 

\makeatletter
\newcommand{\labitem}[2]{%
\def\@itemlabel{(\textbf{#1})}
\item
\def\@currentlabel{\textbf{#1}}\label{#2}}
\makeatother

\usepackage{mathtools}
\usepackage{enumitem}
\setlist[enumerate]{font=\normalfont}

\DeclareMathOperator{\Tho}{Tho}
\DeclareMathOperator{\stab}{stab}
\newcommand{\smashing}{\mathbb{S}}
\newcommand{\thick}{\mathbb{T}}

\newcommand{\infl}{\mathsf{I}}
\newcommand{\cont}{\mathsf{C}}
\newcommand{\rig}{\mathrm{rig}}
\newcommand{\SPB}{\mathrm{Sp}^{\mathrm{Bal}}}

\newcommand{\equalizer}[3]{\xymatrix{{#1}\ar[r] & {#2}\ar@<.5ex>@{->}[r] \ar@<-.5ex>@{->}[r] & {#3}}}
\newcommand{\coequalizer}[3]{\xymatrix{{#1}\ar@<.5ex>@{->}[r] \ar@<-.5ex>@{->}[r] & {#2}\ar[r] & {#3}}}

\title{The telescope conjecture for algebraic stacks}
\date{May 15, 2017}
\author[J. Hall]{Jack Hall}
\address{Department of Mathematics\\University of Arizona\\Tucson, AZ 85721\\USA}
\email{jackhall@math.arizona.edu}
\author[D. Rydh]{David Rydh}
\address{KTH Royal Institute of Technology\\Department of Mathematics\\SE\nobreakdash-100\ 44\ Stockholm\\Sweden}
\email{dary@math.kth.se}
\thanks{The first author was supported by the Australian Research Council DE150101799 while some of this work was completed.}
\thanks{The second author is supported by the Swedish Research Council
2011-5599 and 2015-05554.}


\begin{document}
\begin{abstract}
  Using Balmer--Favi's generalized idempotents, we establish the telescope conjecture for many algebraic stacks. Along the way, we classify the thick tensor ideals of perfect complexes of stacks.
\end{abstract}
\maketitle
\section*{Introduction}
Let $\mathscript{T}$ be a triangulated category. A basic question is the classification of 
the triangulated subcategories of $\mathscript{T}$.  The interest in this question stemmed from the pioneering work of 
Hopkins \cite{MR932260} who connected it to the longstanding telescope conjecture in 
algebraic topology. Over the last two decades, spurred by work of Neeman \cite{MR1174255}, there 
has been interest in this from the perspective of algebraic geometry. Most recently, this has been 
considered by Antieau \cite{MR3161100}, Stevenson \cite{MR3181496}, 
and Balmer--Favi \cite{MR2806103}. We will now briefly recall some positive results in this direction that have been proved by others before we get to the contributions of this article.

Let $X$ be a quasi-compact and quasi-separated scheme. Consider $\PERF(X)$, the triangulated 
category of perfect complexes on $X$. If $P$ is a perfect complex on $X$, then let 
$\supph(P) \subseteq |X|$ be its cohomological support. If $E \subseteq |X|$ is a subset, then define
\[
\mathcal{I}(E) = \bigl\{P \in \PERF(X) \suchthat \supph(P) \subseteq E \bigr\}.
\]
This is a thick tensor ideal of $\PERF(X)$, that is, $\mathcal{I}(E)$ is a full triangulated subcategory of $\PERF(X)$ which is closed under direct summands and by tensoring with perfect complexes. Conversely, given a thick tensor ideal $\mathscript{C}$ of $\PERF(X)$ we obtain a subset
\[
\varphi(\mathscript{C}) = \bigcup_{P \in \mathscript{C}} \supph(P) \subseteq |X|.
\]
The cohomological support of a perfect complex is closed with quasi-compact
complement. Thus, $\varphi(\mathscript{C})$ is a \emph{Thomason} subset---a
union of closed subsets with quasi-compact complements.

Let $\thick(X)$ denote the collection of thick tensor ideals of $\PERF(X)$ and let $\Tho(X)$ 
denote the set of Thomason subsets of $|X|$. Both $\thick(X)$ and 
$\Tho(X)$ are partially ordered via inclusion, so we have described order
preserving maps:
\[
\xymatrix{\Tho(X)  \ar@<0.5ex>[r]^-{\smash{\mathcal{I}}} & \thick(X). \ar@<0.5ex>[l]^-{\varphi} }
\]
These maps are mutually inverse by an important result of
Thomason \cite{MR1436741}. This generalizes earlier work of Hopkins and Neeman.

Now consider $\DQCOH(X)$, the unbounded derived category of quasi-coherent sheaves on 
$X$. A natural class of subcategories of $\DQCOH(X)$ to consider are the \emph{smashing 
tensor ideals}, that is, those smashing Bousfield localizations $\mathscript{S} \subseteq 
\DQCOH(X)$ that are closed under tensor product by objects of $\DQCOH(X)$  (see 
\S\ref{S:tensor_ideals} for details). Let $\smashing(X)$ denote the collection of all smashing 
tensor ideals of $\DQCOH(X)$. It is well-known that there is an 
\emph{inflation} map:
\[
\infl \colon \thick(X) \to \smashing(X),
\]
which sends a thick tensor ideal $\mathscript{C}$ to the smallest localizing subcategory in $\DQCOH(X)$ containing $\mathscript{C}$ (see Corollary \ref{C:tensor_inflation} for details). There is also a \emph{contraction} map:
\[
\cont \colon \smashing(X) \to \thick(X),
\]
which sends a smashing tensor ideal $\mathscript{S}$ to the thick tensor ideal 
$\mathscript{S} \cap \PERF(X)$ of $\PERF(X)$. Thus, we have described order preserving maps:
\[
\xymatrix{\thick(X)  \ar@<0.5ex>[r]^-{\smash{\infl}} & \smashing(X) \ar@<0.5ex>[l]^-{\cont}.}
\]
A beautiful result of Balmer--Favi \cite[Cor.~6.8]{MR2806103}, generalizing the work of 
Neeman, is that these maps are mutually inverse when $X$ is noetherian. Such results have typically been referred to as resolutions of the \emph{tensor telescope conjecture}. We recommend \cite[\S1]{MR3161100} for an excellent discussion of the history of this conjecture as well as a number of positive results in this direction.

It is also natural to consider all of this in the equivariant setting. Let $G$
be an algebraic group acting on a quasi-projective variety $V$. It is then
natural to consider the category $\PERF^G(V)$ of perfect complexes of
$G$-equivariant quasi-coherent sheaves on $V$ and the unbounded derived
category $\DCAT(\QCOH^G(V))$ of $G$-equivariant quasi-coherent sheaves.

The first main result of this article characterizes thick tensor ideals of
$\PERF^G(V)$ and smashing tensor ideals of $\DCAT(\QCOH^G(V))$ in terms of
$G$-invariant subsets of $|V|$.
\begin{maintheorem}\label{main:equivariant_char0}
  Let $V$ be a quasi-projective variety over a field $k$. Let $G$ be an 
  affine algebraic group acting on $V$.
  Then there are mutually inverse maps
  \[
  \xymatrix{\Tho(\mbox{$G$-orbits on $V$})  \ar@<0.5ex>[r]^-{\smash{\mathcal{I}}} & \thick(\PERF^G(V)) 
    \ar@<0.5ex>[l]^-{\varphi} \ar@<0.5ex>[r]^-{\smash{\infl}} & \smashing(\DCAT(\QCOH^G(V))) 
    \ar@<0.5ex>[l]^-{\cont} }
\]
if either
  \begin{enumerate}
  \item the characteristic of $k$ is zero, and $V$ is semi-normal (e.g., normal or smooth) or the action of $G$ on $V$ is linearizable; or
  \item $G^0$ is of multiplicative type (e.g., a torus) and the geometric component 
    group of $G$ has order prime to the characteristic of $k$.
  \end{enumerate}
\end{maintheorem}
When $V$ is affine and $G$ is diagonalizable, Theorem~\ref{main:equivariant_char0} was proved by Dell'Ambrogio and Stevenson~\cite{MR2995031} and they also
classified all localizing tensor ideals. As far as we know,
all other prior results
\cite{MR2927050,MR2570954,hall_balmer_cms} only establishes parts of Theorem \ref{main:equivariant_char0} in the case of finite stabilizers or when the action is trivial \cite{MR3161100}.

Algebraic stacks form the natural setting for Theorem \ref{main:equivariant_char0}. In fact, Theorem \ref{main:equivariant_char0} is a consequence of a much more general and conceptual result, which we will now describe.

Let $X$ be a quasi-compact and quasi-separated algebraic stack. We let $\DQCOH(X)$ denote its unbounded derived category of quasi-coherent
sheaves and let $\DQCOH(X)^c$ denote the thick subcategory of compact objects
of $\DQCOH(X)$~\cite{perfect_complexes_stacks}. If $X$ is a scheme or algebraic
space, then there is an equality $\DQCOH(X)^c = \PERF(X)$. In general,
there is always an inclusion $\DQCOH(X)^c \subseteq \PERF(X)$ but it can be
strict~\cite[Thm.~C]{hallj_dary_alg_groups_classifying}.  When $V$ and $G$ are
as in Theorem~\ref{main:equivariant_char0} and $X=[V/G]$ is the stack quotient,
then $\DQCOH(X)^c = \PERF^G(V)$ and
$\DQCOH(X)=\DCAT(\QCOH^G(V))$~\cite{hallj_neeman_dary_no_compacts}.

Let $\thick(X)$ denote the collection of thick tensor
ideals of $\DQCOH(X)^c$ and let $\smashing(X)$ denote the collection of
smashing tensor ideals of $\DQCOH(X)$.
\begin{maintheorem}\label{main:thick_stack}
  Let $X$ be a quasi-compact and quasi-separated algebraic stack. If $\DQCOH(X)$
  is
  compactly generated, then there are maps:
  \[
  \xymatrix{\Tho(X)  \ar@<0.5ex>[r]^-{\smash{\mathcal{I}_X}} & \thick(X) 
\ar@<0.5ex>[l]^-{\varphi_X} \ar@<0.5ex>[r]^-{\smash{\infl_X}} & \smashing(X) 
\ar@<0.5ex>[l]^-{\cont_X}. }
  \]
  such that $\cont_X\circ \infl_X=\ID{\thick(X)}$.
  If $X$ satisfies the Thomason condition, then $\mathcal{I}_X$ and $\varphi_X$ are 
  mutually inverse. If in addition $X$ is noetherian, then $\infl_X$ and $\cont_X$ are 
  mutually inverse. 
\end{maintheorem}

An algebraic stack $X$ satisfies the \emph{Thomason condition} if
\begin{enumerate}
\item the unbounded derived category of quasi-coherent sheaves
  $\DQCOH(X)$ is compactly generated, and 
\item for every Zariski-closed subset $Z \subseteq |X|$ with quasi-compact complement, there exists
  a perfect complex $P$ on $X$ with $\supph(P) = Z$. 
\end{enumerate}
If $X$ is a quasi-compact and quasi-separated algebraic stack, then $X$ is known to satisfy the Thomason condition in the following situations:
\begin{enumerate}[label=(\alph*)]
\item $X$ is a scheme \cite[Thm.~3.1.1]{MR1996800}; or
\item $X$ has quasi-finite and separated diagonal (e.g., algebraic spaces and Deligne--Mumford stacks) \cite[Thm.~A]{perfect_complexes_stacks}; or
\item \label{main:thick_stack:quotients} $X=[V/G]$, where $V$ is a scheme over a characteristic $0$ field $k$, $G$ is 
  an affine algebraic group acting on $V$ and
  \begin{enumerate}[label=(\alph{enumi}\arabic*)]
  \item $V$ is quasi-affine \cite[Prop.~8.4]{perfect_complexes_stacks}; or
  \item $V$ is quasi-projective over $k$ and the action of $G$ is linearizable \cite[Ex.~7.5 \& Prop.~8.4]{perfect_complexes_stacks}; or
  \item $V$ is quasi-projective over $k$ and semi-normal \cite[Cor.~to Thm.~B]{perfect_complexes_stacks}; or
  \item $V$ is of finite type over $k$ and normal \loccit;
  \end{enumerate}
\item $X$ admits a separated \'etale cover by a stack as above \cite[Thm.~B]{perfect_complexes_stacks}; or
\item $X$ is of finite type over a field with affine diagonal and
  $\stab(x)^0$ is linearly reductive for every closed point $x\in |X|$ (e.g.,
  $X$ admits a good moduli space)~\cite[Thm.~2.26]{AHR_lunafield}.
\end{enumerate}
It is also known that many gerbes satisfy the Thomason condition~\cite[Thm.~3.6]{hallj_dary_alg_groups_classifying}. In particular, we view the Thomason condition as a very mild condition for stacks in characteristic zero and for stacks with linearly reductive stabilizers in positive characteristic.

The equivalence on the left in Theorem \ref{main:thick_stack} was proved for stacks with 
quasi-finite diagonal by the first author in \cite[Thm.~1.1]{hall_balmer_cms} using a completely 
different approach (i.e., tensor nilpotence). The equivalence on the right in Theorem \ref{main:thick_stack} was proved for Deligne--Mumford stacks with locally constant stabilizers admitting coarse moduli schemes by Antieau \cite{MR3161100}. In particular, Theorem \ref{main:thick_stack} generalizes all previously known results, and answers some questions posed by Antieau \cite[Qstn.~6.3 \& 6.8]{MR3161100}.

Recall that an algebraic stack $X$ is \emph{concentrated} if it is quasi-compact, 
quasi-separated and $\Orb_X$ is a compact object of $\DQCOH(X)$, or equivalently, 
$\DQCOH(X)^c = \PERF(X)$ (Lemma \ref{L:rig_compact}). Part of Theorem 
\ref{main:thick_stack}, for concentrated stacks, can be rephrased in terms of the Balmer spectrum \cite{MR2196732}. 
\begin{maintheorem}\label{main:balmer_concentrated}
  Let $X$ be a concentrated algebraic stack. If $X$ satisfies the Thomason condition, 
  then there is a natural isomorphism of locally ringed spaces:
    \[
  (|X|,\Orb_{X_{\mathrm{Zar}}}) \to \SPB(\PERF(X)),
  \]
  where $\Orb_{X_{\mathrm{Zar}}}$ is the Zariski sheaf $U\mapsto \Gamma(U,\Orb_X)$.
\end{maintheorem}
When $X$ has finite stabilizers, Theorem \ref{main:balmer_concentrated} was first proven by the first author~\cite[Thm.~1.2]{hall_balmer_cms}. Recall that quasi-compact algebraic stacks with quasi-finite and separated diagonal satisfy the Thomason condition \cite[Thm.~A]{perfect_complexes_stacks}. Such stacks are concentrated precisely when they are \emph{tame}, e.g., of characteristic zero.

To prove Theorem \ref{main:thick_stack} we had two key insights: the first was that
Balmer--Favi's generalized idempotents \cite{MR2806103} can be used to prove some general injectivity 
results (see \S\ref{S:injectivity}) and the second was that faithfully flat descent of Thomason 
subsets holds (Lemma~\ref{L:ff-descent-of-Thomason}). Combining these two results, we 
prove that the statement of Theorem \ref{main:thick_stack} is local for faithfully flat 
morphisms of finite presentation. We thus reduce the problem to the affine situation, which was treated by Neeman in \cite{MR1174255}.

\subsection*{Assumptions and conventions}
For algebraic stacks, we follow the conventions of \cite{stacks-project}. In particular, 
an algebraic stack is not a priori assumed to have any separation properties. All stacks 
appearing in the article, however, will at least be quasi-separated (i.e., the diagonal 
and second diagonal are quasi-compact). An algebraic stack is  \emph{noetherian} if it 
is quasi-separated and admits a faithfully flat cover by a noetherian scheme. 
\subsection*{Acknowledgements}
We wish to thank Greg Stevenson, Ben Antieau, and Amnon Neeman for their comments and suggestions on the article. We also wish to thank the anonymous referee for their suggestions and pointers to the literature.
\section{Subcategories of triangulated categories}
In this section, we briefly review some well-known results on subcategories of triangulated categories. The key definition being that of a \emph{smashing localization} (see Theorem 
\ref{T:smashing_loc} and \cite[Prop.~5.5.1]{MR2681709}). We found the results and expositions contained in \cite[\S2 and \S9]{MR1812507}, \cite{MR2681709}, and \cite[\S2]{MR2806103} particularly helpful.

Throughout this section we let $\mathscript{T}$ be a triangulated category. Let $\mathscript{S} \subseteq 
\mathscript{T}$ be a triangulated subcategory. By that we will mean the following:
\begin{enumerate}
\item $\mathscript{S}$ is a full subcategory of $\mathscript{T}$; 
\item $\mathscript{S}$ contains $0$; and
\item if $s_1 \to s_2 \to s_3 \to s_1[1]$ is a distinguished triangle in
  $\mathscript{T}$ and two of the $s_i$ belong to $\mathscript{S}$,
  then so does the third.
\end{enumerate}
These conditions imply that $\mathscript{S}$ is a triangulated category and the resulting functor $\mathscript{S} \to \mathscript{T}$ is triangulated. Moreover, these conditions also imply that $\mathscript{S}$ is \emph{replete}, that is, if $s \in \mathscript{S}$ and $t\in \mathscript{T}$ and $s\simeq t$, then $t\in \mathscript{S}$. 
\subsection{Thick and localizing subcategories}
We say that $\mathscript{S}$ is \emph{thick} if 
every $\mathscript{T}$-direct summand of $s\in \mathscript{S}$ belongs to 
$\mathscript{S}$. If $\mathscript{T}$ admits small coproducts, then we say that $\mathscript{S}$ is \emph{localizing} if it also admits small coproducts and the inclusion of $\mathscript{S}$ into $\mathscript{T}$ preserves small coproducts. An Eilenberg swindle argument shows that every localizing subcategory is thick. 

If $f \colon \mathscript{T}
\to \mathscript{T}'$ is a triangulated functor, then $\ker f$ is a thick
subcategory. Conversely, if $\mathscript{S}\subseteq \mathscript{T}$ is a thick
subcategory, then there is a \emph{Verdier} quotient $q \colon \mathscript{T}\to
\mathscript{T}/\mathscript{S}$ and $\mathscript{S}=\ker q$. The quotient need
not have small Hom-sets. Note that if $\mathscript{S}$ is a localizing subcategory of $\mathscript{T}$, then the Verdier quotient $q \colon \mathscript{T}\to
\mathscript{T}/\mathscript{S}$ preserves small coproducts.

The following notation will be useful: if $\mathscript{C} \subseteq \mathscript{T}$ is a 
class, let $\langle \mathscript{C} \rangle$ denote the smallest localizing subcategory containing $\mathscript{C}$ (see \cite[\S 3.2]{MR1812507} for details). We record for future reference the following lemma.
\begin{lemma}\label{L:image_preimage_localizing}
  Let $f\colon \mathscript{T} \to \mathscript{T}'$ be a triangulated functor. Assume 
  $\mathscript{T}$, $\mathscript{T}'$ admit small coproducts and that $f$ preserves 
  them.
  \begin{enumerate}
  \item \label{LI:image_preimage_localizing:preimage} If $\mathscript{S}'$ is a localizing subcategory of $\mathscript{T}'$, then the full 
    subcategory $f^{-1}(\mathscript{S}')$ of $\mathscript{T}$ is a localizing subcategory.
  \item \label{LI:image_preimage_localizing:image} If $\mathscript{C} \subseteq \mathscript{T}$ is a class, then
    $\langle f(\mathscript{C})\rangle = \langle f(\langle
    \mathscript{C} \rangle) \rangle$.
  \end{enumerate}
\end{lemma}
\begin{proof}
    For \itemref{LI:image_preimage_localizing:preimage}, we first show that
    the subcategory $f^{-1}(\mathscript{S}')$ is triangulated.
    Let $s' \to s \to s'' \to s'[1]$ be a triangle in $\mathscript{T}$ such that 
    $f(s')$ and $f(s) \in \mathscript{S}'$. Hence, $f(s') \to f(s) \to f(s'') \to f(s')[1]$ is a 
    distinguished triangle in $\mathscript{T}'$. Since $\mathscript{S}'$ is a triangulated 
    subcategory of $\mathscript{T}'$, it follows that $f(s'') \in \mathscript{S}'$. That 
    $f^{-1}(\mathscript{S}')$ is closed under small coproducts follows immediately from 
    the assumption that $f$ preserves small coproducts. Hence, $f^{-1}(\mathscript{S}')$ is 
    a localizing subcategory.

    For \itemref{LI:image_preimage_localizing:image}: since $f(\mathscript{C}) \subseteq 
    f(\langle \mathscript{C} \rangle)$, it follows 
    immediately that $\langle f(\mathscript{C}) \rangle \subseteq \langle f(\langle 
    \mathscript{C} \rangle) \rangle$. For the reverse inclusion, by 
    \itemref{LI:image_preimage_localizing:preimage}, $f^{-1}(\langle f(\mathscript{C}) 
    \rangle)$ is a localizing subcategory of $\mathscript{T}$. Moreover, it contains 
    $\mathscript{C}$, so it also contains $\langle \mathscript{C} \rangle$. Hence, $f(\langle \mathscript{C} \rangle) \subseteq \langle f(\mathscript{C}) 
    \rangle$ as required.
\end{proof}

\subsection{Compact generation}
We say that $t\in \mathscript{T}$ is \emph{compact} if, for every set of objects $\{ x_i \}_{i\in I}$ in $\mathscript{T}$, the natural map
\[
\bigoplus_{i\in I}\Hom(t,x_i) \to \Hom\Bigl(t,\bigoplus_{i\in I} x_i\Bigr)
\]
is an isomorphism. We let $\mathscript{T}^c$ denote the full subcategory of compact objects of $\mathscript{T}$. It is easily determined that it is a thick subcategory.

We say that $\mathscript{T}$ is \emph{compactly generated} if there is a set $T$ of 
compact objects in $\mathscript{T}$ such that if $x\in \mathscript{T}$ and $\Hom(t[n],x) = 
0$ for all $t\in {T}$ and $n \in \Z$, then $x=0$. We call $T$ a \emph{generating} set for $\mathscript{T}$.

\emph{Brown representability} 
\cite[Thm.~4.1]{MR1308405} says that for a compactly generated triangulated category $\mathscript{T}$ a triangulated functor $F \colon \mathscript{T} \to \mathscript{T}'$ preserves small coproducts if and only if $F$ admits a right adjoint. 

An important and useful result here is \emph{Thomason's Localization Theorem} 
\cite[Thm.~2.1]{MR1308405}, which we now state. Assume $\mathscript{T}$ is compactly 
generated and $S$ is a set of compact objects of 
$\mathscript{T}$. If $\mathscript{S} = \langle S \rangle$, then
\begin{enumerate}
\item $\mathscript{S}$ is compactly generated as a triangulated category by $S$;
\item if $S$ generates $\mathscript{T}$, then $\mathscript{S} = \mathscript{T}$;
\item $\mathscript{S} \cap \mathscript{T}^c = \mathscript{S}^c$ and is the smallest thick subcategory of $\mathscript{T}^c$ containing $S$; and
\item the quotient $\mathscript{T}^c/\mathscript{S}^c \to (\mathscript{T}/\mathscript{S})^c$ is fully faithful with dense image.
\end{enumerate}
Note 
that by ``dense image'' we mean that the thick closure of 
$\mathscript{T}^c/\mathscript{S}^c$ in $(\mathscript{T}/\mathscript{S})^c$ is an equivalence. 

There are more powerful results along these lines using the theory of 
\emph{well-generated} triangulated categories \cite{MR1812507}. We will occasionally mention some results in this level of generality, though for the purposes of this article it is safe to substitute this with ``compactly generated''.

\subsection{Bousfield localization}
A Verdier quotient $q\colon \mathscript{T} \to \mathscript{T}/\mathscript{S}$ is a \emph{Bousfield localization} if $q$ admits a right adjoint $s$.

If $q$ admits a right adjoint $s$, then
$s$ is fully faithful, $\mathscript{S}$ is localizing,
and the quotient has small Hom-sets. If $\mathscript{T}$
is well-generated, then Brown representability implies that $q$ admits a right adjoint if and only if
$\mathscript{S}$ is localizing and the quotient
has small Hom-sets \cite[Ex.\ 8.4.5]{MR1812507}. Conversely, given a triangulated functor
$f \colon \mathscript{T} \to \mathscript{T}'$ that admits a fully faithful
right adjoint, the kernel $\ker f$ is localizing and there is an equivalence of
categories $\mathscript{T}/\ker f\to \mathscript{T}'$. 

The following example summarizes some results from \cite[\S2]{perfect_complexes_stacks} that we will use in this article.
\begin{example}\label{E:stack_pbk_psh}
Let $j\colon U\to X$ be a quasi-compact, quasi-separated and representable (more 
generally, concentrated) morphism of algebraic stacks. Then $\LDERF \QCPBK{j} \colon 
\DQCOH(X) \to \DQCOH(U)$ admits a right adjoint $\RDERF \QCPSH{j} \colon \DQCOH(U) \to \DQCOH(X)$ that preserves 
small coproducts. If in addition $j$ is a flat monomorphism, then
$\RDERF \QCPSH{j}$ is fully faithful. Indeed, the counit $\LDERF \QCPBK{j}\RDERF \QCPSH{j}\to
\ID{U}$ is an isomorphism by flat base change (since $U\times_X
U=U$). Thus, $\LDERF \QCPBK{j} \colon \DQCOH(X) \to \DQCOH(U)$ is a Bousfield
localization.
\end{example}

A related notion is that of a \emph{Bousfield localization functor}, which is a pair 
$(L,\lambda)$ where $L \colon \mathscript{T} \to \mathscript{T}$ is an exact functor and 
$\lambda \colon \ID{\mathscript{T}} \to L$ is a natural transformation such that $L\lambda\colon L\to L^2$ is an 
isomorphism and $L\lambda = \lambda L$.

The \emph{$L$-local} objects are those belonging to the image of $L$; note that
$t\in \mathscript{T}$ is $L$-local if and only if $\lambda_t$ is an
isomorphism. The \emph{$L$-acyclic} objects are those in the kernel of $L$. In
particular, there are no maps from $L$-acyclic to $L$-local objects; this also
characterizes $L$-acyclic and $L$-local objects in terms of the other
collection~\cite[Prop.~4.10.1]{MR2681709}.

The Bousfield localization functors on $\mathscript{T}$ naturally form a category. A \emph{Bousfield colocalization functor} $(\Gamma,\gamma)$ is a Bousfield localization functor on $\mathscript{T}^{\opp}$. 

If $\mathscript{C}$ is a collection of objects of $\mathscript{T}$, then let 
\begin{align*}
  \mathscript{C}^{\perp} = \{ t\in \mathscript{T} \suchthat \Hom(c,t) = 0, \;\forall c\in \mathscript{C} \}.
\end{align*}
We call $\mathscript{C}^{\perp}$ the \emph{right orthogonal} of $\mathscript{C}$ in 
$\mathscript{T}$. The \emph{left orthogonal}, ${}^{\perp}\mathscript{C}$, of 
$\mathscript{C}$ in $\mathscript{T}$ is the right orthogonal of $\mathscript{C}$ in 
$\mathscript{T}^{\opp}$. Here we follow the conventions of \cite[\S 4.8]{MR2681709} 
and \cite[Def.~2.5]{MR2806103}, which is the opposite to \cite[9.1.10/11]{MR1812507}.

The following well-known theorem (see \cite[4.9--12]{MR2681709}) ties these notions together. 
\begin{theorem}[Bousfield localization]\label{T:bousfield_local}
Let $\mathscript{T}$ be a triangulated category and let $\mathscript{S}\subseteq
\mathscript{T}$ be a thick subcategory. The following are equivalent:
\begin{enumerate}
\item The Verdier quotient $q\colon \mathscript{T}\to
  \mathscript{T}/\mathscript{S}$ has a right adjoint.
\item Each $t\in\mathscript{T}$ fits in an exact triangle
\[
t'\to t\to t''\to t'[1]
\]
with $t'\in \mathscript{S}$ and $t''\in \mathscript{S}^{\perp}$.
\item There exists a Bousfield localization functor $(L,\lambda)$ with $\mathscript{S} = \ker(L)$.
\item There exists a Bousfield colocalization functor $(\Gamma,\gamma)$ with $\mathscript{S} = \im(\Gamma)$.
\end{enumerate}
Under these equivalent conditions, the subcategory
$\mathscript{S}=\ker(L)=\im(\Gamma)$
is localizing,
the subcategory $\mathscript{S}^\perp=\ker(\Gamma)=\im(L)$ is colocalizing
(i.e., closed under products),
the triangle is unique,
$t'=\Gamma(t)$, and $t''=L(t)$. In particular, the triangle is functorial. We
denote this triangle by
\[
\Delta_{\mathscript{S}}(t):=\bigl(\Gamma(t)\to t\to L(t)\to \Gamma(t)[1]\bigr).
\]
\end{theorem}
A subcategory $\mathscript{S}$ as in the theorem is called a \emph{Bousfield
subcategory}. A Bousfield subcategory is localizing and the converse holds if
either (1) $\mathscript{S}$ is well-generated or (2) $\mathscript{T}$ is
well-generated and $\mathscript{T}/\mathscript{S}$ has small
Hom-sets~\cite[5.2.1]{MR2681709} or~\cite[Ex.~8.4.5 and
  Prop.~9.1.19]{MR1812507}. If $\mathscript{T}$ is well-generated,
then $\mathscript{S}$ is well-generated if $\mathscript{S}$ has a set of
generators~\cite[Lem.~A.1]{hallj_neeman_dary_no_compacts},
e.g., if $\mathscript{S}$ is the localizing envelope of a set of elements.

If $\mathscript{S}\subseteq \mathscript{S}'$ is an inclusion of Bousfield
subcategories of $\mathscript{T}$, then there is a unique natural
transformation of triangles
$(\epsilon,\ID{},\varphi)\colon \Delta_{\mathscript{S}}\to
\Delta_{\mathscript{S}'}$~\cite[4.11.2]{MR2681709}, \cite[Rmk.~2.9]{MR2806103}.
\begin{example}\label{E:stack_loc_col}
We continue Example \ref{E:stack_pbk_psh}. The localization functor is $L=\RDERF 
\QCPSH{j}\LDERF \QCPBK{j}$ with
the unit $\lambda\colon \ID{X}\to \RDERF 
\QCPSH{j}\LDERF \QCPBK{j}$. The colocalization $\Gamma$ is local cohomology with respect to $X\setminus U$.
\end{example}
\subsection{Smashing localization}
An important special case of Bousfield localizations are 
\emph{smashing} localizations. The following  
well-known theorem (see \cite[Prop.~5.5.1]{MR2681709}), ties several formulations of this condition together.
\begin{theorem}[Smashing localization]\label{T:smashing_loc}
Let $\mathscript{T}$ be a triangulated category and $\mathscript{S}\subseteq
\mathscript{T}$ be a Bousfield subcategory. The following conditions are equivalent:
\begin{enumerate}
\item $L$ preserves small coproducts.
\item $\Gamma$ preserves small coproducts.
\item The right adjoint of $q\colon \mathscript{T}\to \mathscript{T}/\mathscript{S}$ preserves small coproducts.
\item The subcategory $\mathscript{S}^\perp$ is localizing.
\end{enumerate}
\end{theorem}
We say that $\mathscript{S}$ is \emph{smashing} when the equivalent conditions
above are satisfied.
Example \ref{E:stack_loc_col} is a smashing localization.
\subsection{Inflation}
In this subsection, we assume that $\mathscript{T}$ is a compactly generated triangulated category.
\begin{example}
  If $X$ is a quasi-compact and quasi-separated scheme, algebraic space or algebraic 
  stack with quasi-finite diagonal, then $\DQCOH(X)$ is a compactly generated 
  triangulated category. If $X$ is a $\Q$-stack that \'etale-locally is a quotient stack, 
  then $\DQCOH(X)$ is also compactly generated. Every compact object of 
  $\DQCOH(X)$ is a perfect complex, and for schemes, algebraic spaces, and 
  $\Q$-stacks with affine stabilizers the converse holds. This is all discussed in detail in 
  \cite{perfect_complexes_stacks}.
\end{example}
A natural way to produce smashing subcategories in compactly generated triangulated categories is through the process of \emph{inflation}. The following result is a simple consequence of Thomason's Localization Theorem. 
\begin{theorem}\label{T:inflation}
Let $\mathscript{T}$ be a compactly generated triangulated
category.
Let $\mathscript{C}$ be a thick subcategory of $\mathscript{T}^c$. Then
\begin{enumerate}
\item $\langle\mathscript{C}\rangle$ is a smashing subcategory of
  $\mathscript{T}$.
\item $\mathscript{C}=\langle\mathscript{C}\rangle^c
  =\langle\mathscript{C}\rangle\cap \mathscript{T}^c$.
\item $\mathscript{T}^c/\mathscript{C}\to (\mathscript{T}/\mathscript{C})^c$
  is fully faithful with dense image.
\item $\mathscript{T}/\langle\mathscript{C}\rangle$ is compactly generated.
\end{enumerate}
\end{theorem}
\begin{proof}
First note that $\langle\mathscript{C}\rangle$ is well-generated since
$\mathscript{T}$ is well-generated and $\langle\mathscript{C}\rangle$ is
generated by the set $\mathscript{C}$. Thus, $\langle\mathscript{C}\rangle$
is a Bousfield subcategory.
To see that $\langle\mathscript{C}\rangle$ is smashing, it is enough to prove
that $\langle\mathscript{C}\rangle^\perp$ is localizing. We have that
\[
t\in \langle\mathscript{C}\rangle^\perp \iff \Hom(c,t)=0, \;\forall c\in \langle\mathscript{C}\rangle
\iff \Hom(c,t)=0, \;\forall c\in \mathscript{C}.
\]
Since $c$ is compact, it follows that $\langle\mathscript{C}\rangle^\perp$ is
localizing. 

The claims (2) and (3) is Thomason's Localization Theorem (see 
\cite[Thm.~2.1]{MR1191736} or \cite[Thm.~3.12]{perfect_complexes_stacks}). 

For (4), we note that $q\colon \mathscript{T}\to
\mathscript{T}/\langle\mathscript{C}\rangle$ has a right adjoint (Bousfield
localization) that preserves coproducts (smashing localization). Hence $q$
takes compact objects to compact objects and a set of compact generators to a
set of compact generators.
\end{proof}

\section{Tensor ideals of tensor triangulated categories}\label{S:tensor_ideals}
In this section, we consider natural variations of the results of the previous section for tensor triangulated categories.

We recall the following definition from \cite[Def.~1.1]{MR2196732}: a \emph{tensor} triangulated category $(\mathscript{T},\otimes, 1)$ is a triangulated category $\mathscript{T}$ with a tensor product $\otimes\colon \mathscript{T} \times \mathscript{T} \to \mathscript{T}$ which is a triangulated functor in each variable and also makes $\mathscript{T}$ symmetric monoidal with unit $1$. We also require $\otimes$ to preserve those coproducts that exist in $\mathscript{T}$. 

A tensor triangulated functor is a triangulated functor between tensor triangulated categories that preserves the tensor product and the unit. Henceforth, we will typically suppress the $\otimes$ and $1$ from the definition of a tensor triangulated category.

To handle algebraic stacks in positive characteristic, it will be necessary for us to consider triangulated categories with a tensor product but not necessarily a unit. These will be called \emph{non-unital} tensor triangulated categories. A non-unital tensor triangulated functor will just be a triangulated functor between non-unital tensor triangulated categories that preserves the tensor product.

\subsection{Tensor ideals}
Let $\mathscript{T}$ be a non-unital tensor triangulated category. 

Let $\mathscript{S} \subseteq \mathscript{T}$ be a triangulated subcategory. We say that 
$\mathscript{S}$ is a \emph{tensor ideal} of $\mathscript{T}$ if for every $s\in 
\mathscript{S}$ and $t\in \mathscript{T}$, the tensor product $s\otimes t$ belongs to $\mathscript{S}$. 

Let $\thick(\mathscript{T})$ and $\smashing(\mathscript{T})$ denote the classes of thick and smashing, respectively, tensor ideals of $\mathscript{T}$. These classes are partially ordered by inclusion. 

If $F \colon \mathscript{T} \to \mathscript{T}'$ is a  non-unital tensor triangulated functor, then there is a natural induced map
\[
\thick(F) \colon \thick(\mathscript{T}) \to \thick(\mathscript{T}')
\]
that is given by sending a thick tensor ideal of $\mathscript{T}$ to the smallest thick 
tensor ideal of $\mathscript{T}'$ containing its image. A much more delicate result that 
we will discuss in this section is the map $\smashing(F) \colon \smashing(\mathscript{T}) 
\to \smashing(\mathscript{T}')$, which is due to Balmer--Favi. The subtlety is that while it 
is possible to define a map $\mathbb{L}(F) \colon \mathbb{L}(\mathscript{T}) \to 
\mathbb{L}(\mathscript{T}')$, where $\mathbb{L}(\mathscript{T})$ denotes the collection 
of localizing tensor ideals of $\mathscript{T}$ and similarly for 
$\mathbb{L}(\mathscript{T}')$, it is not at all obvious that there are natural conditions one 
could put on $F$ to guarantee that $\mathbb{L}(F)$ would send objects of 
$\smashing(\mathscript{T})$ to $\smashing(\mathscript{T}')$. 

The following notation will be useful: if $\mathscript{C} \subseteq \mathscript{T}$ is a class, let $\langle \mathscript{C} \rangle_{\tensor}$ denote the smallest localizing tensor ideal containing $\mathscript{C}$. 

For the remainder of this section, we assume that $\mathscript{T}$ is a tensor triangulated category (in the sense of Balmer).
\subsection{Thick tensor ideal inflation}
In Theorem \ref{T:inflation}, we saw that Thomason's Localization Theorem provided a 
method to produce smashing subcategories of compactly generated triangulated 
categories. Here we briefly explain why the same process works for tensor triangulated categories. It is most appropriate to view the following result as a corollary to Theorem \ref{T:inflation}. 
\begin{corollary}\label{C:tensor_inflation}
Let $\mathscript{T}$ be a compactly generated tensor triangulated
category. Then there are inclusion preserving maps
\[
\xymatrix{\thick(\mathscript{T}^c) \ar@<.5ex>[r]^-{\smash{\infl_{\mathscript{T}}}}  &
\smashing(\mathscript{T}) \ar@<.5ex>[l]^-{\cont_{\mathscript{T}}},} 
\]
where $\infl_{\mathscript{T}}(\mathscript{C}) = \langle \mathscript{C} \rangle$ and $\cont_{\mathscript{T}}(\mathscript{S}) = \mathscript{S} \cap \mathscript{T}^c$. Moreover, $\cont_{\mathscript{T}} \circ \infl_{\mathscript{T}} = \ID{}$ and $\langle \mathscript{C}\rangle = \langle \mathscript{C} \rangle_{\tensor}$. 
\end{corollary}
\begin{proof}
By Theorem \ref{T:inflation}, it is sufficient to prove that $\langle \mathscript{C} \rangle$ is a tensor ideal. 

First, let $x \in \mathscript{C}$. Consider the full subcategory $\mathscript{T}_x$ of $\mathscript{T}$ whose objects are those $t$ such that $x\otimes t \in \langle \mathscript{C} \rangle$. Clearly, $\mathscript{T}_x$ is closed under distinguished triangles and small coproducts. Moreover since $\mathscript{C}$ is a tensor ideal in $\mathscript{T}^c$, $\mathscript{T}^c \subseteq \mathscript{T}_x$. By Thomason's Localization Theorem, $\mathscript{T}_x = \mathscript{T}$. Hence, if $x\in \mathscript{C}$ and $t\in \mathscript{T}$, then $x\tensor t \in \langle \mathscript{C} \rangle$. 

Now fix $t \in \mathscript{T}$ and consider the full subcategory $\mathscript{C}_t$ of  
$\langle \mathscript{C} \rangle$ whose objects are those $x$ such that $x\otimes t \in 
\langle \mathscript{C} \rangle$. Clearly $\mathscript{C}_t$ is closed under 
distinguished triangles and small coproducts. By the above, we also know that 
$\mathscript{C} \subseteq \mathscript{C}_t$. By Thomason's Localization Theorem 
applied to $\langle \mathscript{C} \rangle$, it follows that $\mathscript{C}_t = \langle 
\mathscript{C} \rangle$. Hence, if $t\in \mathscript{T}$ and $x \in \langle 
\mathscript{C} \rangle$, then $x\otimes t \in \langle \mathscript{C} \rangle$. That is, 
$\langle \mathscript{C}\rangle$ is a tensor ideal.
\end{proof}
The following corollary of Thomason's Localization Theorem connects the telescope conjecture with the tensor telescope conjecture (cf.\ \cite[Lem.~3.2]{MR3161100} and \cite[Cor.~3.11.1(a)]{MR1436741}).
\begin{corollary}\label{C:unit_gen_thick}
  Let $\mathscript{T}$ be a tensor triangulated category that is compactly generated by its 
  unit. Then thick subcategories of $\mathscript{T}^c$ and localizing subcategories
  of $\mathscript{T}$ are tensor ideals. 
\end{corollary}
\begin{proof}
  Let $\mathscript{C}$ be a thick subcategory of $\mathscript{T}^c$ (resp.\ a
  localizing subcategory of $\mathscript{T}$). Let $c\in \mathscript{C}$ and
  consider the subcategory $\mathscript{S}$ of $\mathscript{T}^c$
  (resp.\ $\mathscript{T}$) whose objects are those $t$ such that $t\tensor c
  \in \mathscript{C}$. Obviously, $1 \in \mathscript{S}$. Moreover,
  $\mathscript{S}$ is a thick subcategory of $\mathscript{T}^c$
  (resp.\ a localizing subcategory of $\mathscript{T}$). By Thomason's
  Localization Theorem, $\langle 1 \rangle=\mathscript{T}$ and $\mathscript{T}^c$ is the
  smallest thick subcategory containing $1$. Thus $\mathscript{S} =
  \mathscript{T}^c$ (resp.\ $\mathscript{S} = \mathscript{T}$).
\end{proof}
\subsection{Duals and rigidity}
If $t\in \mathscript{T}$, let $\SHom_{\mathscript{T}}(t,-)$ be a right adjoint to the functor $-\otimes t\colon \mathscript{T} \to \mathscript{T}$ whenever it exists. Typically, $\SHom_{\mathscript{T}}(t,-)$ is referred to as the \emph{internal} $\Hom$. This is for good reasons: if $t' \in \mathscript{T}$, then
\[
\Hom_{\mathscript{T}}(t,t') =  \Hom_{\mathscript{T}}(1\otimes t,t') = \Hom_{\mathscript{T}}(1,\SHom_{\mathscript{T}}(t,t')). 
\]
If an internal $\Hom$ exists at $t$, then for every $t' \in \mathscript{T}$ there is a natural morphism
\begin{equation}\label{E:dual-mult}
\SHom_{\mathscript{T}}(t,1) \otimes t' \to \SHom_{\mathscript{T}}(t,t')
\end{equation}
induced from the counit map $\SHom_{\mathscript{T}}(t,1) \otimes t \to 1$ by
tensoring with $t'$ and adjunction.
We say that $t$ is \emph{rigid} or \emph{strongly dualizable} if the
morphism~\eqref{E:dual-mult} is an isomorphism for every $t'$. We will denote $\SHom_{\mathscript{T}}(t,1)$ by $t^\vee$. Thus, if $t$ is rigid, then the natural morphisms
\[
\Hom_{\mathscript{T}}(x,t^\vee\otimes y) \to
\Hom_{\mathscript{T}}(x\otimes t,y),\quad\text{and}\quad t\to t^{\vee\vee}
\]
are isomorphisms.

Let $\mathscript{T}^{\rig}$ denote the full subcategory of rigid objects of $\mathscript{T}$. The following lemma connects rigid and compact objects, and also appears in \cite[Thm.~2.1.3]{MR1388895}.
\begin{lemma}\label{L:rig_compact}
  Let $\mathscript{T}$ be a tensor triangulated category.
  \begin{enumerate}
  \item \label{L:rig_compact:rig_act_compact} If $t\in \mathscript{T}^c$ and $d\in 
    \mathscript{T}^{\rig}$, then $t\tensor d \in \mathscript{T}^c$. 
  \item\label{L:rig_compact:compact_1} If $1 \in \mathscript{T}^c$, then $\mathscript{T}^{\rig} 
    \subseteq \mathscript{T}^c$.
  \item\label{L:rig_compact:compacts_are_rigid} If $\mathscript{T}^c \subseteq 
    \mathscript{T}^{\rig}$, then the restriction of $\otimes$ on 
    $\mathscript{T}$ to $\mathscript{T}^c$ makes it a non-unital tensor triangulated category.
  \end{enumerate}
\end{lemma}
\begin{proof}
  Claims \itemref{L:rig_compact:compact_1} and \itemref{L:rig_compact:compacts_are_rigid} are 
  immediate consequences of \itemref{L:rig_compact:rig_act_compact}. For  
  \itemref{L:rig_compact:rig_act_compact}, let $\{x_i\}_{i\in I}$ be a set of objects in 
  $\mathscript{T}$. Then the map $\bigoplus_i \Hom(t\tensor d,x_i) \to \Hom(t\tensor d,\bigoplus_i x_i)$ 
  factors as the following sequence of isomorphisms:
  \begin{align*}
    \bigoplus_i \Hom(t\tensor d, x_i) &\cong \bigoplus_i \Hom(t,d^\vee \tensor x_i) \cong \Hom\Bigl(t,\bigoplus_i (d^\vee \tensor x_i)\Bigr) \\
    &\cong \Hom\Bigl(t, d^\vee \otimes \bigoplus_i x_i\Bigr) \cong \Hom\Bigl(t \tensor d, \bigoplus_i x_i\Bigr).
  \end{align*}
  Hence, $t\otimes d \in \mathscript{T}^c$.
\end{proof}

A related notion is the following: $\mathscript{T}$ is \emph{closed} if for each $t \in \mathscript{T}$, an internal $\Hom$ exists. By Neeman's adjoint functor theorem: if $\mathscript{T}$ is well-generated, then $\mathscript{T}$ is always closed. The following lemma comes from \cite[Lem.~3.3.1]{MR1388895}.
\begin{lemma}\label{L:rigid_bousfield}
Let $\mathscript{T}$ be a tensor triangulated category.
Suppose that $\mathcal{S}$ is a Bousfield tensor ideal. Let $L$ denote the
Bousfield localization functor. Then
there is a natural transformation
\[
\alpha(t)\colon L(1)\otimes t\to L(t),
\]
which is an isomorphism when $t$ is rigid.
\end{lemma}
\begin{proof}
Tensor the exact triangle $\Gamma(1)\to 1\to L(1)$ with $t$ and apply $L$.
This gives the exact triangle $L(\Gamma(1)\otimes t)\to L(t)\to L(L(1)\otimes
t)$. Since $\Gamma(1)\otimes t\in \mathcal{S}$, the first object is zero
so we have an
isomorphism $\varphi_t\colon L(t)\homotopic L(L(1)\otimes t)$. Let $\alpha(t)$
be the composition
\[
\alpha(t)\colon L(1)\otimes t\xrightarrow{\lambda_{L(1)\otimes t}} L(L(1)\otimes t)\xrightarrow{\varphi_t^{-1}} L(t).
\]
If $t$
is rigid, then
\[
\Hom(z,L(1)\otimes t)=\Hom(z\otimes t^\vee,L(1))=0
\]
for all $L$-acyclic
$z$, i.e., $z\in \mathcal{S}$. Hence, $L(1)\otimes t$ is $L$-local, so
$\lambda_{L(1)\otimes t}$ is an isomorphism, as is $\alpha(t)$.
\end{proof}
We say that $\mathscript{T}$ is \emph{rigidly compactly generated} if it is compactly generated and $\mathscript{T}^c \subseteq \mathscript{T}^{\rig}$.

The following proposition provides an important characterization of smashing tensor ideals (cf.~\cite[Thm.~2.13]{MR2806103}). 
\begin{proposition}\label{P:characterization_smashing_ideal}
Let $\mathscript{T}$ be a tensor triangulated category. If $\mathscript{S}$ is
a Bousfield tensor ideal, then the following conditions
are equivalent:
\begin{enumerate}
\item \label{PI:characterization_smashing_ideal:ideal} $\mathscript{S}^\perp$ is a tensor ideal.
\item \label{PI:characterization_smashing_ideal:L} There exists an isomorphism of functors $L\simeq L(1)\otimes-$.
\item \label{PI:characterization_smashing_ideal:Gamma} There exists an isomorphism of functors $\Gamma \simeq \Gamma(1) \otimes -$. 
\item \label{PI:characterization_smashing_ideal:Delta} There exists an isomorphism of functors $\Delta\simeq \Delta(1)\otimes-$.
\end{enumerate}
These equivalent conditions imply that $\mathscript{S}$ is smashing. The
converse holds if $\mathscript{T}$ is rigidly compactly generated.
\end{proposition}
\begin{proof}
For \itemref{PI:characterization_smashing_ideal:ideal} $\Rightarrow$ 
\itemref{PI:characterization_smashing_ideal:Delta}: the triangle $\Delta(1) \otimes t$ is uniquely isomorphic to $\Delta(t)$ by Theorem \ref{T:bousfield_local}. Also, \itemref{PI:characterization_smashing_ideal:Delta} easily implies \itemref{PI:characterization_smashing_ideal:L} and \itemref{PI:characterization_smashing_ideal:Gamma}. In particular, $L$ preserves
small coproducts so $\mathscript{S}$ is smashing.

For \itemref{PI:characterization_smashing_ideal:L} $\Rightarrow$ \itemref{PI:characterization_smashing_ideal:ideal}: if $x$, $y\in \mathscript{T}$, then $L(x)\otimes y \simeq L(1) \otimes x\otimes y \simeq L(x\otimes y)$. In particular, $\mathscript{S}^\perp$ is a tensor ideal. 

For \itemref{PI:characterization_smashing_ideal:Gamma} $\Rightarrow$ 
\itemref{PI:characterization_smashing_ideal:ideal}: if $x\in \mathscript{S}^\perp$ and $y\in \mathscript{T}$, then $\Gamma(x\otimes y) \simeq \Gamma(1)  \otimes x\otimes y \simeq \Gamma(x)\otimes y \simeq 0$. In particular, $\mathscript{S}^\perp$ is a tensor ideal. 

Conversely if $\mathscript{S}$ is smashing, then for each $t\in \mathscript{T}$ there is a natural
transformation $\alpha(t)\colon L(1)\otimes t\to L(t)$, which is an isomorphism when $t$ is rigid (Lemma \ref{L:rigid_bousfield}). Since $L$ preserves coproducts and
$\mathscript{T}$ is compactly generated by rigid objects, it follows that
$\alpha$ is an isomorphism for all $t$. 
\end{proof}

\begin{example}
We have seen that Example \ref{E:stack_loc_col} is a smashing localization.
The kernel $\mathscript{S}=\ker L$ is also a tensor ideal. Indeed, the natural map
\[
\alpha(M)\colon L(\Orb_X)\otimes \cplx{M} = (\RDERF \QCPSH{j}\Orb_U)\otimes_{\Orb_X} \cplx{M}
\to \RDERF \QCPSH{j}\LDERF \QCPBK{j}(\cplx{M})=L(\cplx{M})
\]
is an isomorphism by the projection formula~\cite[Cor.~4.12]{perfect_complexes_stacks}.
Thus $L(\cplx{M}\otimes \cplx{N})=L(\cplx{M})\otimes \cplx{N}$ so $\mathscript{S}$ is a tensor ideal.
\end{example}
\subsection{Idempotents}
We conclude this section with a discussion of idempotents in tensor triangulated categories. These have been studied in detail in several places \cite{MR1388895,MR2806103,MR1462832,MR2489634}. For this subsection, we will assume that $\mathscript{T}$ is rigidly compactly generated.

 The 
relevance here is that a smashing tensor ideal is determined by the map $1 \to L(1)$, 
which is a \emph{right idempotent}. Similarly, a smashing tensor ideal is determined 
by the map $\Gamma(1) \to 1$, which is a \emph{left idempotent}. Let us make this 
precise.

An \emph{idempotent triangle} is a distinguished triangle $e \xrightarrow{\gamma} 1 
\xrightarrow{\lambda} f \to e[1]$ such that $e\tensor f \simeq 0$. A \emph{left 
idempotent} is a morphism $\gamma \colon e \to 1$ such that $\gamma \tensor e$ is 
an isomorphism. A \emph{right idempotent} is a morphism $\lambda \colon 1 \to f$ 
such that $\lambda\tensor f$ is an isomorphism. For each type of idempotent, there is a 
natural notion of morphisms between them.

It is easily seen that there are classes $\mathbb{D}(\mathscript{T})$, 
$\mathbb{E}(\mathscript{T})$, and $\mathbb{F}(\mathscript{T})$ with elements the 
isomorphism classes of idempotent triangles, left idempotents, and right idempotents, 
respectively. It is readily observed that these classes are all partially ordered under the 
relation of ``there exists a morphism'' \cite[3.2]{MR2806103}. There are also natural 
maps of partial orders:
\[
 \xymatrix{\mathbb{E}(\mathscript{T}) \ar[r] \ar[dr] & \ar[l] \ar[d] \mathbb{D}(\mathscript{T}) \ar[r] & \ar[l] \ar[dl] \mathbb{F}(\mathscript{T}) \\ & \mathbb{S}(\mathscript{T}). \ar[u] \ar[ul] \ar[ur] & }
\]
The maps to and from $\smashing(\mathscript{T})$ are provided by Proposition \ref{P:characterization_smashing_ideal}: a smashing tensor ideal $\mathscript{S}$
is mapped to the idempotent triangle $\Gamma(1)\to 1\to L(1)$; a left
idempotent $e\to 1$ is mapped to the smashing tensor ideal
$\im(e\otimes -)$; a right idempotent $1\to f$ is mapped to the smashing tensor
ideal $\ker(f\otimes -)$.
A key insight of Balmer--Favi is that these maps are all equivalences \cite[\S3]{MR2806103} of partial orders. In fact, they prove more: 
$\mathbb{D}(\mathscript{T})$ is a lattice, with supremum given by $\tensor$ in 
$\mathbb{F}(\mathscript{T})$ and infimum given by $\tensor$ in 
$\mathbb{E}(\mathscript{T})$. Henceforth, we will use the identification between 
smashing tensor ideals and the various idempotents freely.

A key observation here is the following: if $F \colon \mathscript{T} \to \mathscript{T}'$ 
is a tensor triangulated functor between rigidly compactly generated triangulated categories, then there is a 
lattice homomorphism
\[
\smashing(F) \colon \smashing(\mathscript{T}) \to \smashing(\mathscript{T}'). 
\]
In terms of idempotents, $\smashing(F)$ takes an idempotent triangle 
$e \xrightarrow{\gamma} 1 \xrightarrow{\lambda} f \to e[1]$ to
$F(e) \xrightarrow{F(\gamma)} F(1) \xrightarrow{F(\lambda)} F(f) \to F(e)[1]$.
In terms of smashing tensor ideals, $\smashing(F)$ takes a smashing tensor ideal
$\mathscript{S}$ to $\langle F(\mathscript{S})
\rangle_{\tensor}=\im\bigl(F(e)\otimes -\bigr)$.

The following variant of \cite[Thm.~6.3]{MR2806103} will be important.
\begin{theorem}\label{T:infl_smash_func}
  Let $F \colon \mathscript{T} \to \mathscript{T}'$ be a tensor triangulated functor 
  between rigidly compactly generated tensor triangulated categories that preserves small coproducts.
  If $F(\mathscript{T}^c) \subseteq \mathscript{T}'^c$, then the following 
  diagram commutes:
  \[
  \xymatrix{\thick(\mathscript{T}^c) \ar[d]_{\thick(F^c)} \ar[r]^{\infl_{\mathscript{T}}} & \smashing(\mathscript{T}) \ar[d]^{\smashing(F)} \\
    \thick(\mathscript{T}'^c) \ar[r]^{\infl_{\mathscript{T'}}} &
    \smashing(\mathscript{T'}).}
  \]
\end{theorem}
\begin{proof}
  Let $\mathscript{C} \subseteq \mathscript{T}^c$ be a thick tensor ideal
  and let $\mathscript{C'}=\thick(F^c)(\mathscript{C})$ be the smallest
  thick tensor ideal containing $F(\mathscript{C})$. Corollary 
  \ref{C:tensor_inflation} implies that $\mathscript{S} = \infl(\mathscript{C}) = \langle 
\mathscript{C} \rangle$ and $\mathscript{S}' = \infl(\mathscript{C'}) = \langle 
\mathscript{C'} \rangle = \langle F(\mathscript{C}) \rangle$ are smashing tensor ideals. Also, $\langle F(\mathscript{C}) 
\rangle = \langle F(\langle \mathscript{C} \rangle) \rangle = \langle F(\mathscript{S}) 
\rangle$ (Lemma 
\ref{L:image_preimage_localizing}\itemref{LI:image_preimage_localizing:image}). In 
particular, $\mathscript{S}' = \langle F(\mathscript{C}) 
\rangle_{\otimes} = \smashing(F)(\mathscript{S})$.  
\end{proof}

Note that it is not at all obvious that the corresponding diagram in Theorem \ref{T:infl_smash_func} with $\cont$ instead of $\infl$ commutes. 
\begin{example}\label{E:conc_pbk}
  The condition that $F(\mathscript{T}^c) \subseteq \mathscript{T}'^c$
  is satisfied when the right adjoint of $F$ preserves small
  coproducts \cite[Thm.~5.1]{MR1308405}. In particular, if $f\colon X' \to X$ is a 
  concentrated morphism of algebraic stacks (e.g., quasi-compact, quasi-separated, 
  and representable), then $\LDERF \QCPBK{f} \colon \DQCOH(X) \to 
  \DQCOH(X')$ preserves compact objects. 
\end{example}

\section{The tensor triangulated telescope conjecture}
We are now in a position to state the tensor triangulated telescope conjecture. For background and some recent progress on this conjecture see \cite{MR3161100}. 

So 
let $\mathscript{T}$ be a rigidly compactly generated tensor triangulated category. Recall that $\thick(\mathscript{T}^c)$ is the class of thick tensor ideals of $\mathscript{T}^c$ and 
$\smashing(\mathscript{T})$ is the class of smashing tensor ideals of 
$\mathscript{T}$. By Corollary \ref{C:tensor_inflation}, there is an order preserving map
\[
\infl_{\mathscript{T}} \colon \thick(\mathscript{T}^c) \to \smashing(\mathscript{T}).
\]
 \begin{definition}[Tensor triangulated telescope conjecture~{\cite[Def.~4.2]{MR2806103}}]
We say that the \emph{tensor triangulated telescope conjecture} holds for $\mathscript{T}$ if $\infl_{\mathscript{T}}$ is bijective.
\end{definition}
Note that Corollary \ref{C:tensor_inflation} proves that $\infl_{\mathscript{T}}$ is injective. In particular, the difficulty lies in establishing the surjectivity of $\infl_{\mathscript{T}}$. By Corollary \ref{C:unit_gen_thick}, the tensor triangulated telescope conjecture for $\DCAT(R)$ is just the usual telescope conjecture. There are examples of non-noetherian rings $R$ for which the telescope conjecture fails for $\DCAT(R)$ \cite{MR1285956}. 
\section{Injectivity results}\label{S:injectivity}
In this section, we now prove some new results. We would like to emphasize that they are really very straightforward from the perspective afforded by idempotents.
\begin{theorem}\label{T:idempotent_injective}
  Let $F \colon \mathscript{T} \to \mathscript{T}'$ be a tensor triangulated functor 
  between rigidly compactly generated tensor triangulated categories. If $F$ is 
  conservative, then 
  \[
  \smashing(F) \colon \smashing(\mathscript{T}) \to \smashing(\mathscript{T'})
  \]
  is injective. 
\end{theorem}
\begin{proof}
  We prove this using idempotents: consider idempotent triangles $\Delta=(e 
\xrightarrow{\gamma} 1 \xrightarrow{\lambda} f \to e[1])$ and $\Delta'=(e' 
\xrightarrow{\gamma'} 1 \to f' \xrightarrow{\lambda'} e'[1])$ such that 
$\smashing(F)(\Delta) = \smashing(F)(\Delta')$. Since $\smashing(\mathscript{T})$ is 
a lattice, it is sufficient to prove that $\Delta=\Delta \wedge \Delta'$ (by symmetry, the result follows). Now the idempotent corresponding to $\Delta \wedge \Delta'$ is $e \tensor e' \xrightarrow{\gamma \tensor \gamma'} 1$ and there is a morphism of idempotents $e \tensor e' \xrightarrow{e \tensor \gamma'} e$. By assumption, $F(e\tensor \gamma')$ is an isomorphism. Since $F$ is conservative, $e\tensor \gamma'$ is an isomorphism. It follows that $\Delta = \Delta \wedge \Delta'$ and we have the claim.
\end{proof}
An immediate consequence of Theorems \ref{T:infl_smash_func}, \ref{T:idempotent_injective} and Corollary \ref{C:tensor_inflation} is the following.
\begin{corollary}\label{C:thick_injective}
  Let $F \colon \mathscript{T} \to \mathscript{T}'$ be a tensor triangulated functor 
  between rigidly compactly generated tensor triangulated categories that preserves small coproducts. If $F$ is 
  conservative and $F(\mathscript{T}^c) \subseteq \mathscript{T}'^c$, then 
  \[
  \thick(F^c) \colon \thick(\mathscript{T}^c) \to \thick(\mathscript{T'}^c)
  \]
  is injective. 
\end{corollary}

\newcommand{\Subsets}{\mathcal{P}}
\newcommand{\Op}{\mathrm{Op}}
\newcommand{\Cl}{\mathrm{Cl}}
\newcommand{\qc}{\mathrm{qc}}
\newcommand{\cons}{\mathrm{cons}}
\section{Topologies}
In this section, we briefly recall some results on the constructible topology on algebraic 
stacks (cf.\ \cite[\S 5]{MR1771927}, \cite[\S IV.1.8--9]{EGA} and~\cite{MR2679038}).

Let $X$ be a quasi-separated algebraic stack. Let $E\subseteq |X|$ be a
subset. We say that $E$ is:
\begin{itemize}
\item \emph{retrocompact} if for every quasi-compact open $U \subseteq |X|$, $E \cap U$ is quasi-compact;
\item \emph{globally constructible} if $E$ is a finite
  union of subsets of the form $W_1\cap {(X\setminus W_2)}$ where
  $W_1$ and $W_2$ are retro-compact open subsets of $X$;
\item \emph{constructible} if $E$ is globally constructible Zariski-locally on $X$;
\item \emph{C-open} if $E$ is open in the constructible topology, i.e.,
  ind-constructible, i.e., a union of constructible subsets;
\item \emph{C-closed} if $E$ is closed in the constructible topology, i.e.,
  pro-constructible, i.e., an intersection of constructible subsets;
\item \emph{S-open} if $E$ is stable under generizations; and
\item \emph{S-closed} if $E$ is stable under specializations.
\end{itemize}
Every morphism of algebraic stacks is continuous in both the C-topology and the S-topology.
Moving constructible sets from an algebraic stack to another is typically accomplished by Chevalley's Theorem, which we now recall.
\begin{theorem}[Chevalley's Theorem]\label{L:chevalley}
Let $f\colon Z\to X$ be a morphism between quasi-separated algebraic stacks.
\begin{enumerate}
\item\label{TI:Chevalley:image} If $f$ is of finite
  presentation, then $f(Z)$ is constructible.
\item\label{TI:Chevalley:cons} If $f$ is of finite
  presentation and $E\subseteq |Z|$ is constructible, then $f(E)$ is
  constructible.
\item\label{TI:Chevalley:ind-cons} If $f$ is locally of
  finite presentation and $E\subseteq |Z|$ is ind-constructible, then $f(E)$
  is ind-constructible.
\end{enumerate}
\end{theorem}
\begin{proof}\footnote{The proof of~\cite[Thm.~5.9.4]{MR1771927} appears to be
    incomplete as only locally closed constructible sets are considered in the
    proof of Cor.~5.9.2 (b).}
The statements are Zariski-local on $X$, so we may assume that $X$ is
quasi-compact. 

\itemref{TI:Chevalley:image}
There is a flattening
stratification of $f$, that is, a sequence of quasi-compact open substacks $\emptyset=U_0\subset U_1\subset U_2\subset\dots\subset U_n=X$ such that $f$ is flat over
$\red{(U_i\setminus U_{i-1})}$.
Indeed, the existence of a flattening stratification can be checked
smooth-locally on $X$. We may thus reduce the question to $X$ affine, and then,
by approximation, to $X$ noetherian.
When $X$ is noetherian, generic flatness produces a flattening
stratification.

Finally, we note that $f(Z)\cap (U_i\setminus U_{i-1})$ is open and
quasi-compact in $(U_i\setminus U_{i-1})$ hence globally constructible. It follows
that $f(Z)$ is globally constructible.

\itemref{TI:Chevalley:cons}
Let $p\colon U\to Z$ be a presentation with $U$ affine. Since $p^{-1}(E)$ is
constructible, there exists a finitely presented morphism $g\colon W\to U$ such
that $g(W)=p^{-1}(E)$. It follows from \itemref{TI:Chevalley:image} that
$f(E)=(f\circ p\circ g)(W)$ is constructible.

\itemref{TI:Chevalley:ind-cons}
Taking a presentation, we may assume that $Z$ is a disjoint union of affine
schemes and the result follows from \itemref{TI:Chevalley:cons}.
\end{proof}

Using Chevalley's theorem we can lift the usual results on the constructible topology
from schemes to quasi-separated algebraic stacks. In particular,
a subset is open (resp.\ closed) in the Zariski topology if and only if it is
C-open and S-open (resp.\ C-closed and S-closed). A subset is constructible
if and only if it is C-open and C-closed.

Arbitrary unions of S-closed sets are S-closed. We can therefore dualize the
S-topology and interchange the roles of closed and open subsets. The (Hochster)
\emph{dual
  topology} $X^*$ on $X$ is the topology where a subset is open if and only if
it is S-closed and C-open.

\begin{lemma}\label{L:dual_top}
Let $X$ be a quasi-compact and quasi-separated algebraic stack. A subset $E\subseteq |X|$ is:
\begin{enumerate}
\item \label{L:dual_top:open} $*$-open if and only if $E$ is a union of constructible closed subsets; and
\item \label{L:dual_top:closed}$*$-closed if and only if $E$ is an intersection of quasi-compact open
  subsets.
\end{enumerate}
\end{lemma}
\begin{proof}
A closed subset is constructible if and only if its complement is
quasi-compact.  It is thus enough to prove \itemref{L:dual_top:closed}. A quasi-compact open subset is
S-open and C-closed, hence so is any intersection. Conversely, if $E$ is S-open
and C-closed, then $E$ is stable under generizations and quasi-compact. The
former implies that $E$ is an intersection of open subsets $E_\alpha$. The
latter implies that we may assume that the $E_\alpha$'s are quasi-compact.
\end{proof}

The $*$-open subsets are also known as \emph{Thomason} subsets. We let $\Tho(X)$ denote the set of Thomason subsets of $X$.

Let $f\colon X\to Y$ be a continuous map between topological spaces. If $f$ is
surjective and either open or closed, then $f$ is \emph{submersive}, i.e., a
subset $E\subseteq Y$ is open if and only if $f^{-1}(E)\subseteq X$ is
open. Thus, if $f\colon X\to Y$ is a morphism of quasi-separated algebraic
stacks that is
\begin{enumerate}
\item surjective,
\item either S-open or S-closed, and
\item either C-open or C-closed,
\end{enumerate}
then $f$ is submersive in both the S-topology and the C-topology, hence also
in the Zariski topology and the dual topology.
Examples of such morphisms are those that are either:
\begin{itemize}
\item proper and surjective, or
\item faithfully flat and locally of finite presentation (Chevalley's Theorem).
\end{itemize}
Another important class of morphisms that are S-submersive and C-submersive
are \emph{subtrusive} morphisms~\cite{MR2679038}.

\begin{lemma}\label{L:ff-descent-of-Thomason}
Let $f\colon X'\to X$ be a morphism of quasi-separated algebraic stacks
that is S-submersive and C-submersive, e.g., $f$ is
faithfully flat and locally of finite presentation. Then $f$ is $*$-submersive.
In particular, if $X''=X'\times_X X'$, then the following sequence of sets is equalizing:
\[
\equalizer{\Tho(X)}{\Tho(X')}{\Tho(X'').}
\]
\end{lemma}
\begin{proof}
$f$ is $*$-submersive by the definition of the dual topology.
Since $f$ is surjective and $|X''|\to |X'|\times_{|X|} |X'|$ is surjective,
the sequence
\[
\equalizer{\Subsets(|X|)}{\Subsets(|X'|)}{\Subsets(|X''|)}
\]
of the sets of all subsets is equalizing. Since $f$ is $*$-submersive, it
follows that $E\subseteq |X|$ is $*$-open, i.e., Thomason, if and only if
$f^{-1}(E)$ is $*$-open.
\end{proof}

\begin{remark}
Spectra of quasi-compact and quasi-separated schemes and stacks are spectral
spaces~\cite{MR0251026}. Stone duality provides an equivalence of categories
between
\begin{enumerate}
\item spectral spaces and spectral maps;
\item coherent frames and coherent maps; and
\item bounded distributive lattices and bounded lattice maps.
\end{enumerate}
Given a spectral space $X$, the corresponding coherent frame is the lattice
$\Op(X)$ of open subsets and the corresponding bounded distributive lattice is
the lattice $\Op_\qc(X)$ of quasi-compact open subsets. Hochster duality takes
the lattice $\Op_\qc(X)$ to the opposite lattice $\Cl_\cons(X)$ of closed and
constructible subsets, the coherent frame $\Op(X)$ to the coherent frame
$\Tho(X)$ and the spectral space $X$ to the spectral space $X^*$.

If $f\colon X'\to X$ is a submersive morphism of spectral spaces, then
\[
\coequalizer{X'\times_X X'}{X'}{X}
\]
is a coequalizer in the category of topological spaces, hence also a
coequalizer in the category of spectral spaces. By Stone and Hochster duality,
we obtain a coequalizer
\[
\coequalizer{X'^*\times_{X^*} X'^*}{X'^*}{X^*}
\]
in the category of spectral spaces. However, this need not be a coequalizer
in the category of topological spaces, that is, $f^*\colon X'^*\to X^*$ need
not be submersive. Indeed, there exists a universally submersive morphism
$f\colon X'\to X$ of affine schemes that is neither S-submersive, nor
$*$-submersive. This example will appear in a forthcoming paper.
\end{remark}

\section{The classification of thick tensor ideals for algebraic stacks}
Let $X$ be a quasi-compact and quasi-separated algebraic stack. In this section, we extend one of the main results of \cite[Thm.~1]{hall_balmer_cms} on the classification of thick tensor ideals in $\DQCOH(X)^c$ to include a number of stacks with infinite stabilizers. So let $X$ be an algebraic stack. It will be convenient to set
\[
\thick(X) := \thick(\DQCOH(X)^c)
\]
Recall that if $M \in \DQCOH(X)$, then 
\[
\supph(M) = \cup_{n\in \Z} \supp(\COHO{n}(M)) \subseteq |X|.
\]
Moreover, if $P$ is a perfect complex on $X$, then $\supph(P)$ is a closed constructible subset of $|X|$ \cite[Lem.~4.8(3)]{perfect_complexes_stacks}. 
\begin{definition}
Let $X$ be a quasi-compact and quasi-separated algebraic stack. We say that $X$:
\begin{enumerate}
\item is \emph{tensor supported} if for every
  $P,Q\in \DQCOH(X)^c$ such that $\supph(P)\subseteq \supph(Q)$ it
  holds that
  $\langle P\rangle_\otimes\subseteq \langle
  Q\rangle_\otimes$;
\item is \emph{compactly supported} if for every closed
  constructible subset $Z$, there exists a compact object
  $P\in\DQCOH(X)^c$ with $\supph(P)=Z$; and 
\item satisfies the \emph{Thomason condition} if it is compactly supported and $\DQCOH(X)$ is compactly generated.
\end{enumerate}
\end{definition}
\begin{example}\label{E:tensor_supp}
  Quasi-compact and quasi-separated schemes are tensor supported \cite[Lem.~3.14]{MR1436741}. So
  are quasi-compact algebraic stacks with quasi-finite and separated
  diagonals~\cite[Lem.~3.1]{hall_balmer_cms}. These results were proved using the ``Tensor 
  Nilpotence Theorem with parameters''.
\end{example}
\begin{example}\label{E:compact_supp}
  Quasi-compact and quasi-separated schemes satisfy the Thomason condition. So do 
  quasi-compact algebraic 
  stacks with quasi-finite and separated diagonal \cite[Thm.~A]{perfect_complexes_stacks}. If $X$ is a 
  $\Q$-stack that \'etale-locally is a quotient stack, 
  then $X$ satisfies the Thomason condition. This is all discussed in detail in 
  \cite{perfect_complexes_stacks} (and see the list in the Introduction to the present article).
\end{example}
Recall the following: given a subset $E \subseteq |X|$ there is a subcategory
\[
\mathcal{I}_X(E) = \{ P \in \DQCOH(X)^c \suchthat \supph(P) \subseteq E \}.
\]
It is readily seen that $\mathcal{I}_X(E)$ is a thick tensor ideal of $\DQCOH(X)^c$. Also if $\mathscript{C} \subseteq \DQCOH(X)^c$ is a thick tensor ideal of $\DQCOH(X)^c$, then let 
\[
\varphi_X(\mathscript{C}) = \cup_{P \in \mathscript{C}} \supph(P) \subseteq |X|.
\]
Since every compact object of $\DQCOH(X)$ is perfect \cite[Lem.~4.4(1)]{perfect_complexes_stacks}, $\varphi_X(\mathscript{C})$ is always a Thomason subset of $|X|$. The following simple lemma relates these notions.
\begin{lemma}\label{L:tens_comp}
  Let $X$ be a quasi-compact and quasi-separated algebraic stack.
  The maps
\[
\xymatrix{\Tho(X)\ar@<0.5ex>[r]^-{\smash{\mathcal{I}_X}} & \ar@<0.5ex>[l]^-{\varphi_X} \thick(X) }
\]
  are adjoint functors of posets.

  \begin{enumerate}
  \item $X$ is tensor supported if and only if $\mathcal{I}_X \circ 
    \varphi_X = \ID{}$. 
    Equivalently, either $\varphi_X$ is injective or $\mathcal{I}_X$ is surjective.
    \item $X$ is compactly supported if and only if $\varphi_X \circ \mathcal{I}_X = \ID{}$. Equivalently, either $\varphi_X$ is surjective
      or $\mathcal{I}_X$ is injective.
  \end{enumerate}
\end{lemma}
\begin{proof}
  Let $E\in\Tho(X)$, then clearly $E\supseteq \varphi_X(\mathcal{I}_X(E))$.
  Conversely, if $\mathscript{C}\in\thick(X)$, then $\mathscript{C}\subseteq
  \mathcal{I}_X(\varphi_X(\mathscript{C}))$. Moreover,
  $\varphi_X\circ \mathcal{I}_X\circ \varphi_X=\varphi_X$ and
  $\mathcal{I}_X\circ \varphi_X\circ \mathcal{I}_X=\mathcal{I}_X$. In particular, 
  $\varphi$ and $\mathcal{I}$ are adjoint functors of posets, a so-called
  Galois connection.

  It is now easily seen that (1) $X$ is tensor supported if and only if $\varphi_X$ is injective; and that (2) $X$ is compactly supported if and only if $\varphi_X$ is surjective on 
  Thomason subsets. 
\end{proof}
\begin{lemma}\label{L:more_thick_commuting}
Let $f\colon X'\to X$ be a concentrated morphism between quasi-compact and quasi-separated algebraic stacks. Then in the diagram
\[
\xymatrix{\Tho(X) \ar[d]_{f^{-1}}\ar@<0.5ex>[r]^-{\smash{\mathcal{I}_X}} & \ar@<0.5ex>[l]^-{\varphi_X} \thick(X) \ar[d]^{\thick((\LDERF \QCPBK{f})^c)}\\
\Tho(X') \ar@<0.5ex>[r]^-{\mathcal{I}_{X'}} & \ar@<0.5ex>[l]^-{\varphi_{X'}} \thick(X') }
\]
\begin{enumerate}
\item the square with $\varphi$ always commutes; and
\item the square with $\mathcal{I}$ commutes when $X'$ is
  tensor supported and $X$ is compactly supported.
\end{enumerate}
\end{lemma}
\begin{proof}
Note that since $f$ is concentrated, the right vertical map is well-defined
(Example~\ref{E:conc_pbk}).
For (1), note that $\supph(\LDERF f^*\cplx{P})\otimes \cplx{Q}\subseteq
\supph(\LDERF f^*\cplx{P})=f^{-1}(\supph \cplx{P})$ and similarly for direct summands. Also, 
(2) follows from (1) since $\varphi_X\circ \mathcal{I}_X=\ID{}$ and
$\mathcal{I}_{X'} \circ \varphi_{X'} = \ID{}$ (Lemma \ref{L:tens_comp}).
\end{proof}
We now come to the main result of this section.
\begin{theorem}\label{T:classification_thick}
Let $X$ be a quasi-compact and quasi-separated algebraic stack. 
If $X$ satisfies the Thomason condition, then $X$ is tensor supported. In particular, 
\[
\xymatrix{\Tho(X) \ar@<0.5ex>[r]^-{\smash{\mathcal{I}_X}} & \ar@<0.5ex>[l]^-{\varphi_X} \thick(X)}
\]
are mutually inverse. 
\end{theorem}
\begin{proof}
  By Lemma \ref{L:tens_comp}, it remains to prove that $\mathcal{I}_X \circ \varphi_X = 
  \ID{}$. Since $X$ is quasi-compact, there exists a smooth and surjective morphism $f\colon X' 
  \to X$, where $X'$ is an affine scheme. By \cite[Lem.~A.3]{MR1174255}, $X'$ is tensor supported.  The quasi-separatedness of $X$ guarantees that 
  $f$ is quasi-compact, quasi-separated, and representable; in particular, concentrated 
  \cite[Lem.~2.5(3)]{perfect_complexes_stacks}. By Lemma 
  \ref{L:more_thick_commuting}, all squares in the diagram:
  \[
  \xymatrix{\Tho(X) \ar[d]_{f^{-1}}\ar@<0.5ex>[r]^-{\smash{\mathcal{I}_X}} & \
    \ar@<0.5ex>[l]^-{\varphi_X} \thick(X) \ar[d]^{\thick((\LDERF f^*)^c)}\\
    \Tho(X') \ar@<0.5ex>[r]^-{\mathcal{I}_{X'}} & \ar@<0.5ex>[l]^-{\varphi_{X'}} \thick(X') }
  \]
  commute. Hence,
  \[
 \thick((\LDERF f^*)^c)\circ \mathcal{I}_X \circ \varphi_X = \mathcal{I}_{X'} \circ (f^{-1}) \circ \varphi_X  = \mathcal{I}_{X'} \circ \varphi_{X'} \circ \thick((\LDERF f^*)^c) = \thick((\LDERF f^*)^c),
  \]
  with the last equality provided by Lemma \ref{L:tens_comp}. But $\LDERF \QCPBK{f}$ preserves compacts and is conservative, so  
  Corollary \ref{C:thick_injective} implies that $\thick((\LDERF f^*)^c)$ is injective. Hence, 
  $\mathcal{I}_X \circ \varphi_X = \ID{}$ as required.
\end{proof}
We can now prove Theorem \ref{main:balmer_concentrated}, which determines the Balmer spectrum of concentrated stacks that satisfy the Thomason condition.
\begin{proof}[Proof of Theorem \ref{main:balmer_concentrated}]
  The proof of \cite[Thm.~1.2]{hall_balmer_cms} applies verbatim using
  Theorem~\ref{T:classification_thick}.
\end{proof}

\section{The tensor triangulated telescope conjecture for algebraic stacks}
The main result of this section is that a large class of noetherian algebraic stacks satisfy the tensor triangulated telescope conjecture. We will prove a more general result, which we expect to be of use in the non-noetherian situation. 

Some notation: if $X$ is a quasi-compact and quasi-separated stack, let 
\[
\smashing(X) := \smashing(\DQCOH(X)) \quad \mbox{and} \quad \infl_X := \infl_{\DQCOH(X)} \colon \thick(X) \to \smashing(X).
\]
Also, we will say that $X$ is \emph{telescoping} if the tensor triangulated telescope conjecture holds for $\DQCOH(X)$, that is, $\infl_X$ is bijective.

\begin{theorem}\label{T:smashing_local}
  Let $f\colon X'\to X$ be a faithfully flat and concentrated morphism of
  finite presentation between quasi-compact and quasi-separated algebraic
  stacks that satisfy the Thomason condition. If $X'$ is telescoping, then $X$
  is telescoping.
\end{theorem}
\begin{proof}
  Let $p\colon X'' \to X'\times_X X'$ be a smooth surjection, where $X''$ is an affine scheme.
  Recall that $X''$ also satisfies the Thomason condition. Let 
  $r$, $s \colon X'' \to X'$ be the two projections. Then there 
  is an induced commutative diagram:
  \[
  \xymatrix@C6pc{
    \Tho(X) \ar[r]^{\mathcal{I}_X}_{\simeq} \ar[d]_{f^{-1}}
    & \thick(X) \ar[r]^{\infl_X} \ar[d]_{\thick([\LDERF \QCPBK{f}]^c)}
    & \smashing(X) \ar[d]_{\smashing(\LDERF \QCPBK{f})} \\
    \Tho(X') \ar[r]^{\mathcal{I}_{X'}}_{\simeq} \ar@<-.5ex>[d]_{r^{-1}} \ar@<.5ex>[d]^{s^{-1}}
    & \thick(X') \ar[r]^{\infl_{X'}} \ar@<-.5ex>[d]_{\thick([\LDERF \QCPBK{r}]^c)} \ar@<.5ex>[d]^{\thick([\LDERF \QCPBK{s}]^c)}
    & \smashing(X') \ar@<-.5ex>[d]_{\smashing(\LDERF \QCPBK{r})} \ar@<.5ex>[d]^{\smashing(\LDERF \QCPBK{s})} \\
    \Tho(X'') \ar[r]^{\mathcal{I}_{X''}}_{\simeq}
    & \thick(X'') \ar[r]^{\infl_{X''}}
    & \smashing(X''). }
  \]
  Note that since $f$, $r$ and $s$ are concentrated, the middle vertical maps are well-defined
  (Example~\ref{E:conc_pbk}) and the squares to the right commute (Theorem \ref{T:infl_smash_func}).
  By Theorem \ref{T:classification_thick}, the maps 
  $\mathcal{I}_X$, $\mathcal{I}_{X'}$ and $\mathcal{I}_{X''}$ are bijective.

  It remains to prove that $\infl_X$ is bijective. By Corollary 
  \ref{C:tensor_inflation}, we already know that $\infl_X$ is injective.
  Since $X'$ is telescoping, $\infl_{X'}$ is bijective, and
  since $X''$ is affine, $\infl_{X''}$ is injective
  (Corollary \ref{C:tensor_inflation}).
  Since $\Tho(X'\times_X X') \to \Tho(X'')$ is injective, 
  the leftmost sequence is exact (Lemma \ref{L:ff-descent-of-Thomason}).
  Moreover, since $\LDERF\QCPBK{f}$ is conservative,
  $\smashing(\LDERF \QCPBK{f})$ is injective (Theorem \ref{T:idempotent_injective}). 
A routine diagram 
  chase proves the result.
\end{proof}
We can now prove the main results of the article.
\begin{proof}[Proof of Theorem \ref{main:thick_stack}]
  The maps $\infl_X$ and $\cont_X$ exist by Corollary~\ref{C:tensor_inflation}.
  If $X$ has the Thomason property, then $\mathcal{I}_X$ and $\varphi_X$ are
  inverses by Theorem \ref{T:classification_thick}. Let $X'\to X$ be a presentation
  with $X'$ affine. If $X$ is noetherian, then $X'$ is
  telescoping by the result of Hopkins and Neeman~\cite{MR1174255}. It follows
  that $X$ is telescoping, that is, $\infl_X$ and $\cont_X$ are inverses, by
  Theorem~\ref{T:smashing_local}.
\end{proof}
\begin{proof}[Proof of Theorem \ref{main:equivariant_char0}]
  Under either assumption, the stack $[V/G]$ satisfies the Thomason condition
  by~\cite[Cor.~9.2]{perfect_complexes_stacks}.
  In particular, $\DQCOH([V/G])$ is compactly generated and thus
  $\DCAT(\QCOH^G(V)) \simeq \DQCOH([V/G])$~\cite[Thm.~1.2]{hallj_neeman_dary_no_compacts}.
  Also $BG$, and hence
  $[V/G]$ is concentrated~\cite[Thm.~B]{hallj_dary_alg_groups_classifying}
  so $\PERF^G(V) \simeq \DQCOH([V/G])^c$. We conclude by
  Theorem \ref{main:thick_stack}.
\end{proof}
\bibliography{references}

\providecommand{\MR}{\relax\ifhmode\unskip\space\fi MR }
\providecommand{\MRhref}[2]{%
  \href{http://www.ams.org/mathscinet-getitem?mr=#1}{#2}
}
\providecommand{\href}[2]{#2}
\begin{thebibliography}{HNR14}

\bibitem[AHR15]{AHR_lunafield}
J.~Alper, J.~Hall, and D.~Rydh, \emph{{A Luna \'etale slice theorem for
  algebraic stacks}}, April 2015,
  \href{http://arXiv.org/abs/1504.06467}{\mbox{arXiv:1504.06467}}.

\bibitem[Ant14]{MR3161100}
B.~Antieau, \emph{A local-global principle for the telescope conjecture}, Adv.
  Math. \textbf{254} (2014), 280--299.

\bibitem[Bal05]{MR2196732}
P.~Balmer, \emph{The spectrum of prime ideals in tensor triangulated
  categories}, J. Reine Angew. Math. \textbf{588} (2005), 149--168.

\bibitem[BB03]{MR1996800}
A.~Bondal and M.~Van~den Bergh, \emph{Generators and representability of
  functors in commutative and noncommutative geometry}, Mosc. Math. J.
  \textbf{3} (2003), no.~1, 1--36, 258.

\bibitem[BF11]{MR2806103}
P.~Balmer and G.~Favi, \emph{Generalized tensor idempotents and the telescope
  conjecture}, Proc. Lond. Math. Soc. (3) \textbf{102} (2011), no.~6,
  1161--1185.

\bibitem[BIK08]{MR2489634}
D.~Benson, S.~B. Iyengar, and H.~Krause, \emph{Local cohomology and support for
  triangulated categories}, Ann. Sci. \'Ec. Norm. Sup\'er. (4) \textbf{41}
  (2008), no.~4, 573--619.

\bibitem[DM12]{MR2927050}
U.~V. Dubey and V.~M. Mallick, \emph{Spectrum of some triangulated categories},
  J. Algebra \textbf{364} (2012), 90--118.

\bibitem[DS13]{MR2995031}
Ivo Dell'Ambrogio and Greg Stevenson, \emph{On the derived category of a graded
  commutative {N}oetherian ring}, J. Algebra \textbf{373} (2013), 356--376.

\bibitem[EGA]{EGA}
A.~Grothendieck, \emph{\'{E}l\'ements de g\'eom\'etrie alg\'ebrique}, I.H.E.S.
  Publ. Math. \textbf{4, 8, 11, 17, 20, 24, 28, 32} (1960, 1961, 1961, 1963,
  1964, 1965, 1966, 1967).

\bibitem[Hal16]{hall_balmer_cms}
J.~Hall, \emph{The {B}almer spectrum of a tame stack}, Ann. K-Theory \textbf{1}
  (2016), no.~3, 259--274.

\bibitem[HNR14]{hallj_neeman_dary_no_compacts}
J.~Hall, A.~Neeman, and D.~Rydh, \emph{One positive and two negative results
  for derived categories of algebraic stacks}, preprint, May 2014,
  \href{http://arXiv.org/abs/1405.1888v2}{\mbox{arXiv:1405.1888v2}}.

\bibitem[Hoc69]{MR0251026}
M.~Hochster, \emph{Prime ideal structure in commutative rings}, Trans. Amer.
  Math. Soc. \textbf{142} (1969), 43--60.

\bibitem[Hop87]{MR932260}
M.~J. Hopkins, \emph{Global methods in homotopy theory}, Homotopy theory
  ({D}urham, 1985), London Math. Soc. Lecture Note Ser., vol. 117, Cambridge
  Univ. Press, Cambridge, 1987, pp.~73--96.

\bibitem[HPS97]{MR1388895}
M.~Hovey, J.~H. Palmieri, and N.~P. Strickland, \emph{Axiomatic stable homotopy
  theory}, Mem. Amer. Math. Soc. \textbf{128} (1997), no.~610, x+114.

\bibitem[HR15]{hallj_dary_alg_groups_classifying}
J.~Hall and D.~Rydh, \emph{Algebraic groups and compact generation of their
  derived categories of representations}, Indiana Univ. Math. J. \textbf{64}
  (2015), no.~6, 1903--1923.

\bibitem[HR17]{perfect_complexes_stacks}
J.~Hall and D.~Rydh, \emph{Perfect complexes on algebraic stacks}, Compositio
  Math. (2017), accepted for publication.

\bibitem[Kel94]{MR1285956}
B.~Keller, \emph{A remark on the generalized smashing conjecture}, Manuscripta
  Math. \textbf{84} (1994), no.~2, 193--198.

\bibitem[Kra10]{MR2681709}
H.~Krause, \emph{Localization theory for triangulated categories}, Triangulated
  categories, London Math. Soc. Lecture Note Ser., vol. 375, Cambridge Univ.
  Press, Cambridge, 2010, pp.~161--235.

\bibitem[Kri09]{MR2570954}
A.~Krishna, \emph{Perfect complexes on {D}eligne-{M}umford stacks and
  applications}, J. K-Theory \textbf{4} (2009), no.~3, 559--603.

\bibitem[LMB]{MR1771927}
G.~Laumon and L.~Moret-Bailly, \emph{Champs alg\'ebriques}, Ergebnisse der
  Mathematik und ihrer Grenzgebiete. 3. Folge., vol.~39, Springer-Verlag,
  Berlin, 2000.

\bibitem[Nee92a]{MR1174255}
A.~Neeman, \emph{The chromatic tower for {$D(R)$}}, Topology \textbf{31}
  (1992), no.~3, 519--532, With an appendix by Marcel B{\"o}kstedt.

\bibitem[Nee92b]{MR1191736}
A.~Neeman, \emph{The connection between the {$K$}-theory localization theorem
  of {T}homason, {T}robaugh and {Y}ao and the smashing subcategories of
  {B}ousfield and {R}avenel}, Ann. Sci. \'Ecole Norm. Sup. (4) \textbf{25}
  (1992), no.~5, 547--566.

\bibitem[Nee96]{MR1308405}
A.~Neeman, \emph{The {G}rothendieck duality theorem via {B}ousfield's
  techniques and {B}rown representability}, J. Amer. Math. Soc. \textbf{9}
  (1996), no.~1, 205--236.

\bibitem[Nee01]{MR1812507}
A.~Neeman, \emph{Triangulated categories}, Annals of Mathematics Studies, vol.
  148, Princeton University Press, Princeton, NJ, 2001.

\bibitem[Ric97]{MR1462832}
J.~Rickard, \emph{Idempotent modules in the stable category}, J. London Math.
  Soc. (2) \textbf{56} (1997), no.~1, 149--170.

\bibitem[Ryd10]{MR2679038}
D.~Rydh, \emph{Submersions and effective descent of \'etale morphisms}, Bull.
  Soc. Math. France \textbf{138} (2010), no.~2, 181--230.

\bibitem[Stacks]{stacks-project}
The {Stacks Project Authors}, \emph{{S}tacks {P}roject},
  \url{http://stacks.math.columbia.edu}.

\bibitem[Ste13]{MR3181496}
G.~Stevenson, \emph{Support theory via actions of tensor triangulated
  categories}, J. Reine Angew. Math. \textbf{681} (2013), 219--254.

\bibitem[Tho97]{MR1436741}
R.~W. Thomason, \emph{The classification of triangulated subcategories},
  Compositio Math. \textbf{105} (1997), no.~1, 1--27.

\end{thebibliography}
\bibliographystyle{bibstyle}
\end{document}